\documentclass[12pt]{article}
\usepackage{latexsym,amssymb,upref,amsmath,amsthm, amsfonts,authblk}
\usepackage{amssymb,amsmath,amsthm, calc, graphicx}
\usepackage{epsfig}
\usepackage{breqn}
\usepackage{footnpag}
\usepackage{rotating}
\usepackage{amsfonts}
\usepackage{setspace}
\usepackage{fullpage}
\usepackage{enumitem}
\usepackage{bbold}
\usepackage{comment}
\usepackage{pgf,tikz}
\usepackage{mathrsfs}
\usetikzlibrary{arrows}
\usepackage{yhmath}
\usepackage{hyperref}
\usepackage{authblk}
\usepackage{mathrsfs}

\newtheorem{thm}{Theorem}

\newtheorem{lemma}[thm]{Lemma}
\newtheorem{conjecture}[thm]{Conjecture}

\newtheorem{defn}[thm]{Definition}

\newtheorem{kov}[thm]{Corollary}

\newtheorem{clm}[thm]{Claim}

\theoremstyle{remark}
\newtheorem{rem}[thm]{Remark}

\newcommand\cB{{\mathcal B}}

\newcommand\cF{{\mathcal F}}

\newcommand\cH{{\mathcal H}}

\def\lc{\left\lceil}   
\def\rc{\right\rceil}

\def\lf{\left\lfloor}   
\def\rf{\right\rfloor}

\newcommand{\abs}[1]{\left\lvert{#1}\right\rvert}

\newcommand{\ignore}[1]{}

\title{Asymptotics for the Tur\'an number of Berge-$K_{2,t}$}

\begin{document}

\author{
D\'aniel Gerbner \thanks{Alfr\'ed R\'enyi Institute of Mathematics, Hungarian Academy of Sciences. e-mail: gerbner@renyi.hu}
\qquad
Abhishek Methuku \thanks{ Central European University, Budapest. e-mail: abhishekmethuku@gmail.com}
\qquad
M\'at\'e Vizer \thanks{Alfr\'ed R\'enyi Institute of Mathematics, Hungarian Academy of Sciences. e-mail: vizermate@gmail.com.}}

\date{
\today}

\maketitle

\begin{abstract}
Let $F$ be a graph. A hypergraph is called \textit{Berge-$F$} if it can be obtained by replacing each edge in $F$ by a hyperedge containing it. 

Let $\mathcal{F}$ be a family of graphs. The \textit{Tur\'an number} of the family Berge-$\mathcal{F}$ is the maximum possible number of edges in an $r$-uniform hypergraph on $n$ vertices containing no Berge-$F$ as a subhypergraph (for every $F \in \mathcal{F}$) and is denoted by $ex_r(n,\mathcal{F})$.
We determine the asymptotics for the Tur\'an number of Berge-$K_{2,t}$ by showing  
$$ex_3(n,K_{2,t})=(1+o(1))\frac{1}{6}(t-1)^{3/2} \cdot n^{3/2}$$ for any given $t \ge 7$. We study the analogous  question for  linear hypergraphs and show
$$ex_3(n,\{C_2, K_{2,t}\}) =  (1+o_{t}(1))\frac{1}{6}\sqrt{t-1} \cdot n^{3/2}.$$

We also prove general upper and lower bounds on the Tur\'an numbers of a class of graphs including $ex_r(n, K_{2,t})$, $ex_r(n,\{C_2, K_{2,t}\})$, and $ex_r(n, C_{2k})$ for $r \ge 3$. Our bounds improve the results of Gerbner and Palmer \cite{gp1}, F\"uredi and \"Ozkahya \cite{FO2017}, Timmons \cite{T2016}, and provide a new proof of a result of Jiang and Ma \cite{JM2016}.
\end{abstract}

\vspace{4mm}

\noindent
{\bf Keywords:} Tur\'an number, Berge hypergraph, bipartite graph  

\noindent
{\bf AMS Subj.\ Class.\ (2010)}: 05C35, 05D99

\section{Introduction}


Tur\'an-type extremal problems in graphs and hypergraphs are the central topic of extremal combinatorics and has a vast literature. For a survey of recent results we refer the reader to \cite{FS2013, K2011,MV2016}.

The classical definition of a hypergraph cycle due to Berge is the following. A \textit{Berge cycle of length $k$} denoted Berge-$C_k$ is an alternating sequence of distinct vertices and distinct hyperedges of the form $v_1$,$h_{1}$,$v_2$,$h_{2},\ldots,v_k$,$h_{k}$,$v_1$ where $v_i,v_{i+1} \in h_{i}$ for each $i \in \{1,2,\ldots,k-1\}$ and $v_k,v_1 \in h_{k}$. A Berge-path is defined similarly. Note that in this setting it makes sense to talk about $C_2$: a Berge-$C_2$ is a Berge copy of the multigraph consisting of two parallel edges, i.e., a hypergraph consiting of two hyperedges that share at least 2 vertices. An important Tur\'an-type extremal result for Berge cycles is due to Lazebnik and Verstra\"ete \cite{LV2003}, who studied the maximum number of hyperedges in an $r$-uniform hypergraph containing no Berge cycle of length less than five (i.e., girth five). Interestingly, they relate the question of estimating the maximum number of edges in a hypergraph of given girth with the famous question of estimating generalized Tur\'an numbers initiated by Brown, Erd\H{o}s and S\'os \cite{BES} and show that the two problems are equivalent in some cases. Since then Tur\'an-type  extremal problems for  hypergraphs in the Berge sense have attracted considerable attention: see e.g., \cite{BGY2008, Missing_case, FO2017, gp1, gp2, GYKL2017, GYL2012, gyl12, JM2016, T2016}.

\vspace{1mm}

Gerbner and Palmer \cite{gp1} gave the following natural generalization of the definitions of Berge cycles and Berge paths. Let $F=(V(F),E(F))$ be a graph and $\cB=(V(\cB),E(\cB))$ be a hypergraph. We say $\cB$ is \textit{Berge-F} if there is a bijection $\phi:E(F) \rightarrow E(\cB)$ such that $e \subseteq \phi(e)$ for all $e \in E(F)$. In other words, given a graph $F$ we can obtain a Berge-$F$ by replacing each edge of $F$ with a hyperedge that contains it. Given a family of graphs $\mathcal F$, we say that a hypergraph $\cH$ is \textit{Berge-$\mathcal F$-free} if for every $F \in \mathcal F$, the hypergraph $\cH$ does not contain a Berge-$F$ as a subhypergraph.
The maximum possible number of hyperedges in a Berge-$\mathcal F$-free hypergraph on $n$ vertices is the \emph{Tur\'an number} of Berge-$\mathcal F$. 

\vspace{1mm}

Only a handful of results are known about the asymptotic behaviour of Tur\'an numbers for hypergraphs. Our main goal in this paper is to determine sharp asymptotics for the Tur\'an number of Berge-$K_{2,t}$ and Berge-$\{C_2, K_{2,t}\}$ in $3$-uniform hypergraphs. In fact, we prove a general theorem which also provides bounds in the $r$-uniform case.

\vspace{1mm}

A topic that is closely related to Berge hypergraphs is \emph{expansions of graphs}. Let $F$ be a fixed graph and let $r \ge 3$ be a given integer. The \textit{$r$-uniform expansion of $F$} is the $r$-uniform hypergraph $F^+$ obtained from $F$ by adding $r-2$ new vertices to each edge of $F$ which are disjoint from $V(F)$ such that distinct vertices are added to distinct edges of $F$. This notion generalizes the notion of a loose cycle for example. The Tur\'an number of $F^+$ is the maximum number of edges in an $r$-uniform hypergraph on $n$ vertices that does not contain $F^+$ as a subhypergraph. In \cite{KMV1,KMV2,KMV3}, Kostochka, Mubayi and Verstra\"ete studied expansions of paths, cycles, trees, bipartite graphs and other graphs. Of particular interest to this paper is their result showing that the Tur\'an number of $K_{2,t}^+$ is asymptotically equal to $\binom{n}{2}$. Interestingly, as we will show in this paper, the asymptotic behavior of the Tur\'an number of Berge-$K_{2,t}$ is quite different.

\vspace{1mm}

Throughout the paper we consider \textit{simple} hypergraphs, which means there are no duplicate hyperedges and we use the term \textit{linear} for a hypergraph if any two different hyperedges contain at most one common vertex (observe that a hypergraph is linear if and only if it is Berge-$C_2$-free). Note that there is some ambiguity around these words in the theory of hypergraphs. Some authors use the word `simple' for hypergraphs that we call linear. For ease of notation sometimes we consider a hypergraph as a set of hyperedges. The \textit{degree} $d(v)$ of a vertex $v$ in a hypergraph is the number of hyperedges containing it.

\vspace{1mm}

A related Tur\'an-type question for graphs is to determine the maximum possible number of cliques of size $r$ in an $F$-free graph. To be able to connect the previously mentioned Tur\'an-type problems and state our results we introduce the following definitions. 

\vspace{1mm}

\begin{defn} For a set of simple graphs $\cF$ and $n \ge r \ge 2$, let
$$ex_r(n,\cF):=\max\{|\cH|: \cH\subset \binom{[n]}{r} \textnormal{ is Berge-$F$-free for all }F \in \cF\}.$$
\noindent
If $\cF=\{F\}$, then instead of $ex_r(n,\{F\})$, we simply write $ex_r(n,F)$.
\end{defn}

\begin{defn}
For $r \ge 2$ and a simple graph $G$ let us denote by $N(K_r, G)$ the number of (different) copies of $K_r$ in $G$. For a simple graph $F$ and $n \ge r \ge 2$, let
$$ex(n,K_r,F):= \max \{ N(K_r,G) : G \textrm{ is an F-free, simple graph on $n$ vertices} \}.$$ 

\noindent
If $r=2$ we simply use $ex(n,F)$ in both cases.  For a bipartite graph $F$, the bipartite Tur\'an number  $ex(m,n,F)$ is the maximum number of edges in  an $F$-free bipartite  graph  with $m$ and $n$ vertices in its color classes.
\end{defn}

\subsubsection*{Structure of the paper and notation} In Section \ref{history} we briefly survey the literature and highlight the results that we improve. In Section \ref{our_results} we state our results and prove some of our corollaries. In Section \ref{BergeF} we provide proofs of our theorems about $r$-uniform, Berge-$F$-free hypergraphs, while in Section \ref{linear} we provide proofs of our theorems about linear, $r$-uniform, Berge-$K_{2,t}$-free hypergraphs. Finally we provide some remarks and connections with other topics.

\vspace{1.5mm}

Throughout the paper we use standard order notions. When it is ambiguous, we write the parameter(s) that the constant depends on, as a subscript. 
\vspace{1.5mm}

In most cases we use $F$ to denote the forbidden graph, $G$ for the base graph and $\cH$ for the hypergraph. 

%
%

\section{History, known results}
\label{history}

One of the most important results concerning the Tur\'an number of complete bipartite graphs is due to K\H{o}v\'ari, S\'os and Tur\'an \cite{KST1954}, who showed that $ex(n, K_{s,t}) = O (n^{2-1/s})$, where $s \le t$. Koll\'ar, R\'onyai and Szab\'o \cite{KRSZ1996}  provided a lower bound matching the order of magnitude, when $t>s!$. (Later Alon, R\'onyai and Szab\'o \cite{ARS1999} provided a matching lower bound for $t>(s-1)!$.)

\vspace{1.5mm}

For $s=2$, F\"uredi proved the following nice result determining the asymptotics for the Tur\'an number of $K_{2,t}$.  

\begin{thm}[\cite{F1996}]
\label{K2tgraph}
For any fixed $t \ge 2$,  we have  $$ex(n, K_{2,t}) = \frac{\sqrt{t-1}}{2} \cdot n^{3/2} + O(n^{4/3})$$
\end{thm}

\noindent
Moreover, he also determined the asymptotics for the balanced case of the bipartite Tur\'an number of  $K_{2,t}$.

\begin{thm}[\cite{F1996}]
\label{biparitieturanK2tgraph}
For any fixed $t \ge 2$,  we have
$$ex(n,n, K_{2,t}) = \sqrt{t-1} \cdot n^{3/2}  + O(n^{4/3}).$$
\end{thm}

\noindent
Now we list a couple of useful results that are needed later. Alon and Shikhelman determined the asymptotics of $ex(n,K_3, K_{2,t})$.

\begin{thm}[\cite{ALS2016}]\label{ask2t} For $t \ge 2$, we have $$ex(n, K_3, K_{2,t})=\frac{1}{6}(t-1)^{3/2}n^{3/2}(1 + o(1)).$$
\end{thm}

\noindent
Recently Luo determined $ex(n,K_r,P_k)$ exactly.

\begin{thm}[\cite{L2017}] \label{Luo} For $n \ge k\ge 2$ and $r\ge 1$ we have $$ex(n, K_r, P_k)=\frac{n}{k-1}\binom{k-1}{r}.$$

\end{thm}

\subsection{Tur\'an-type results for Berge-F-free hypergraphs}

In this section we briefly survey results concerning Berge-$F$-free hypergraphs, focusing mainly on the results that we improve in our article.

\vspace{1mm}

One of the first results concerning Tur\'an numbers of Berge cycles is due to Lazebnik and Verstra\"ete \cite{LV2003} who showed that $ex_3(n, \{C_2,C_3,C_4\}) = n^{3/2}/6 + o(n^{3/2}).$ Very recently this was strengthened by Ergemlidze, Gy\H{o}ri, and Methuku \cite{BEM} showing that $ex_3(n, \{C_2,C_3,C_4\}) \sim ex_3(n, \{C_2,C_4\}).$
Bollob\'as and Gy\H{o}ri \cite{BGY2008} estimated the Tur\'an number of Berge-$C_5$ by showing $n^{3/2}/{3\sqrt{3}} \le ex_3(n, C_5) \le \sqrt{2}n^{3/2} + 4.5n$.  Very recently, this estimate was improved by Ergemlidze, Gy\H{o}ri, Methuku \cite{C_5}. In \cite{BEM} they also considered the analogous question for linear hypergraphs and proved that $ex_3(n, \{C_2, C_5\}) = n^{3/2}/{3\sqrt{3}}+o(n^{3/2})$. Surprisingly, even though the lower bound here is the same as the lower bound in the Bollob\'as-Gy\H{o}ri theorem, the hypergraph they construct in order to establish their lower bound is very different from the hypergraph used in the Bollob\'as-Gy\H{o}ri theorem. The latter is far from being linear.
\vspace{1mm}

Gy\H{o}ri and Lemons \cite{GYL2012} generalized the Bollob\'as-Gy\H{o}ri theorem to Berge cycles of any given odd length and proved that $ex_3(n, C_{2k+1}) \le 4k^4n^{1+1/k}+15k^4n+10k^2n$ for all $n, k \ge 1$.
Note that this upper bound has the same order of magnitude as the upper bound on the maximum possible number of edges in a $C_{2k}$-free graph (see the Even Cycle Theorem of Bondy and Simonovits \cite{BS1974}). This shows the surprising fact that the maximum number of hyperedges in a Berge-$C_{2k+1}$-free hypergraph is significantly different from the maximum possible number of edges in a $C_{2k+1}$-free graph. Recently, F\"uredi and \"Ozkahya \cite{FO2017} improved this result by showing $ex_3(n,C_{2k+1}) \le (9k^2 + 10k + 5)n^{1+1/k} + O(k^2n)$.


\vspace{2mm}

For Berge cycles of even length, F\"uredi and \"Ozkahya proved the following bound.

\begin{thm}[\cite{FO2017}]\label{foc2k} For $k \ge 2$, we have
$$ex_3(n,C_{2k})\le \frac{2k}{3}ex(n,C_{2k}),$$
and
$$ex(n,K_3,C_{2k})\le \frac{2k-3}{3}ex(n,C_{2k}).$$
\end{thm}

\noindent
We improve the first inequality in Corollary \ref{even_cycle} by showing $\frac{2k}{3}$ can be replaced by $\frac{2k-3}{3}$ provided $k \ge 5$.

\vspace{4mm}

\noindent
Gy\H{o}ri and Lemons \cite{GyLemonsruniform} also showed that for general $r$-uniform hypergraphs with $r \ge 4$, $ex_r(n,C_{2k+1}) \le O(k^{r-2}) \cdot ex_3(n,C_{2k+1})$ and $ex_r(n,C_{2k}) \le O(k^{r-1}) \cdot ex(n,C_{2k})$. Jiang and Ma \cite{JM2016} improved these results by an $\Omega(k)$ factor. In particular, for the even cycle case they showed the following.
\begin{thm}[\cite{JM2016}]\label{jiangma} If $n,k \ge 1$ and $r \ge 4$, then we have
$$ex_r(n,C_{2k}) \le O_r(k^{r-2}) \cdot ex(n,C_{2k}).$$
\end{thm}
\noindent
We give a new proof of the above result in Corollary \ref{jiangma2} (with a better constant factor).

\vspace{4mm}

\noindent
Gerbner and Palmer proved the following about $r$-uniform Berge $K_{2,t}$-free hypergraphs.

\begin{thm}[\cite{gp1}]\label{gp1} If $t\le r-2$, then $$ex_r(n,K_{2,t})=O(n^{3/2}),$$ 
and if $t= r-2$, then $$ex_r(n,K_{2,t})=\Theta(n^{3/2}).$$
\end{thm} 

\noindent
We extend this result to other ranges of $t$ and $r$, and prove more precise bounds in Corollary  \ref{runifbergek2t}.

\vspace{4mm}
\noindent

Timmons studied the same problem for linear hypergraphs and proved the following nice result. 
\begin{thm}[\cite{T2016}]
\label{Timmons_Palmer}
For all $r \ge 3$ and $t \ge 1$,  we have
$$ex_r(n, \{C_2, K_{2,t}\})  \le \frac{\sqrt{2(t+1)}}{r}n^{3/2} +  \frac{n}{r}.$$
Let $r\ge 3$ be an integer and $l$ be any integer with $2l + 1 \ge r$. If $q\ge 2lr^3$
is a power of an odd prime and $n = rq^2$, then
$$ex_r(n, \{C_2, K_{2,t}\}) \ge\frac{l}{r^{3/2}}n^{3/2}+O(n),$$ 
where $t-1=(r-1)(2l^2-l)$.
\end{thm}

\noindent
Note that Timmons mentioned that the upper bound was pointed out by Palmer using methods similar to the ones used in \cite{T2016}. 

\noindent
We improve Theorem \ref{Timmons_Palmer} in  Theorem \ref{linearberge}, Theorem \ref{linearlower} and Theorem \ref{linearlower2}.

\vspace{4mm}

\noindent

Finally we present a simple but useful result of Gerbner and Palmer that connects $ex(n,K_r,F)$ and $ex_r(n,F)$. We include its proof for the sake of completeness.

\begin{thm}[\cite{gp2}]\label{gp2} For $r \ge 2$ and any graph $F$, we have
$$ex(n,K_r,F)\le ex_r(n,F)\le ex(n,K_r,F)+ex(n,F).$$
\end{thm}
\begin{proof}
 If we are given an $F$-free graph $G$, let us replace each clique of size $r$ in it by a hyperedge. The resulting hypergraph is obviously Berge-$F$-free. This proves the first inequality in the theorem.

Now we prove the second inequality. Suppose we are given a Berge-$F$-free hypergraph $\cH$. We order its hyperedges $h_1, h_2, \dots, h_m$ arbitrarily and perform the following procedure. From each hyperedge $h_i$ we pick exactly one pair $xy \subset h_i$ and color it \emph{blue} only if $xy$ has not already been colored before (by the hyperedges $h_1, h_2, \dots, h_{i-1}$). If all the pairs in $h_i$ have been colored blue already then we say that the hyperedge $h_i$ is blue.  It is easy to see that after this procedure, the number of  hyperedges in $\cH$ is  equal to the number of blue pairs  plus blue hyperedges. Moreover, since $\cH$ is Berge-$F$-free, the graph of blue pairs, $G$, is $F$-free and the number of blue hyperedges is at most the number of $K_r$'s in $G$ as all the pairs contained in a blue hyperedge are blue, finishing the proof.
\end{proof}

\vspace{1.5mm}

\noindent
One of our main results shown in the next section, determines the asymptotics for the Tur\'an number of Berge-$K_{2,t}$ for $t\ge 7$ in the case that $r = 3$. Note that Theorem \ref{gp2} combined with Theorem \ref{ask2t} gives $$(1 + o(1))\frac{1}{6}(t-1)^{3/2}n^{3/2} \le ex_3(n,K_{2,t}) \le (1 + o(1))\frac{1}{6}(t-1)^{3/2}n^{3/2} + (1 + o(1))\frac{\sqrt{t-1}}{2}n^{3/2}.$$ Thus we have an upper bound which differs from the lower bound by $$ex(n, K_{2,t})=  (1+o(1)) \frac{\sqrt{t-1}}{2}n^{3/2}.$$ This has the same order of magnitude in $n$ and a lower order of magnitude in $t$ compared to the lower bound. However, the simple idea used in the proof of Theorem \ref{gp2} is not useful to reduce this gap, we will introduce new ideas. Our main focus in this paper is to determine sharp asymptotics.

\section{Our results} 
\label{our_results}

\subsection{A general theorem}

First we state a general theorem that applies to many graphs and not just $K_{2,t}$. For convenience of notation in the rest of the paper, let us define $$f(r)=
\left(\binom{r}{2}-2\right)\left(1+\left(\binom{r}{2}-1\right)\left(\binom{r}{2}-3\right)\right)$$ and $$g(r)=
f(r)+\binom{r}{2}-2.$$ 

\vspace{2mm}
\begin{thm}\label{main} Let $F$ be a $K_r$-free graph. Let $F'$ be a graph we get by deleting a vertex from $F$, and assume there are constants $c$ and $i$ with $0\le i\le r-1$ with $ex(n, K_{r-1},F')\le cn^{i}$ for every $n$. 

\vspace{3mm}

(a) If $cn^{i-1}\ge rg(r)/2,$ then we have $$ex_r(n,F)\le 2c\frac{ex(n,F)n^{i-1}}{r}.$$

\vspace{2mm}

(b) If $cn^{i-1}\le rg(r)/2,$ then we have $$ex_r(n,F)\le g(r)\cdot ex(n,F).$$

(c) If $i>1$ and $n$ is large enough, then we have
$$ex_r(n,F) \le c(r-1)ex(n,F)^{i} \left(\frac{2}{n}\right)^{i-1}.$$

\end{thm}

\vspace{3mm} 

\noindent

\begin{rem}\label{lini} The proof of the above theorem can be modified to show that if $F$ contains $K_r$, similar upper bounds hold with slightly different multiplicative constant factors. 
Theorem \ref{main} together with these inequalities show that if every cycle in $F$ contains the same vertex $v$ for some $v \in V(F)$, then we have $$ex_r(n,F)=O(ex(n,F))$$
for every $r\ge 3$.
\end{rem}

\vspace{2mm}

\noindent
Also note that $g(3)=2$. Moreover, if $F=K_{2,t}$ and $F'=K_{1,t}$, or if $F=C_{t+2}$ and $F'=P_{t+1}$, then we have $$ex(n, K_2, F') = ex(n,F')\le (t-1)\frac{n}{2}.$$
Therefore, using Theorem \ref{main} part (a) with $c = \frac{t-1}{2}$ and $i = 1$ implies the upper bounds in Corollary \ref{maincor} and Corollary \ref{even_cycle} which are given below.

\subsubsection{Asymptotics for Berge-$K_{2,t}$}

\begin{kov}\label{maincor} Let $t\ge 7$. Then $$ex_3(n,K_{2,t})=\frac{1}{6}(t-1)^{3/2}n^{3/2}(1+o(1)).$$

\end{kov}

\noindent
Our lower bound in the above result follows from Theorem \ref{gp2} and Theorem \ref{ask2t}.
The latter theorem considers the $K_{2,t}$-free graph $G$ constructed by F\"uredi in Theorem \ref{K2tgraph} and shows that the number of triangles in it is at least $\frac{1}{6}(t-1)^{3/2}n^{3/2}(1+o(1))$. Replacing each triangle in $G$ by a hyperedge on the same vertex set, we get a Berge-$K_{2,t}$-free hypergraph containing the desired number of hyperedges.

Below we show the analogous result for general $r$-uniform hypergraphs that is sharp in the order of magnitude of $n$.

\begin{kov}\label{runifbergek2t} (a) If $t > \lc \frac{r}{2} \rc-2 \ge 0$, then we have 

$$(1+o(1))\frac{\sqrt{t-1}}{r^{3/2}}n^{3/2} \le ex_r(n,K_{2,t}).$$

\vspace{2mm}

(b) If $\frac{\binom{t}{r-1}}{t} \ge rg(r)/2$, then we have

$$ex_r(n,K_{2,t}) \le (1+o(1)) \frac{\sqrt{(t-1)}\binom{t}{r-1}}{r \cdot t} n^{3/2} .$$

If $\frac{\binom{t}{r-1}}{t} \le rg(r)/2$, then we have

$$ex_r(n,K_{2,t}) \le  (1+o(1)) g(r)\sqrt{t-1}n^{3/2}.$$

\end{kov}

\noindent
This result improves Theorem \ref{gp1}. 

\begin{rem}
If $t \le 6$, then Theorem \ref{main} gives $$ex_3(n,K_{2,t}) \le \sqrt{t-1} \cdot n^{3/2}$$ and the lower bound in Corollary \ref{maincor} still holds. On the other hand putting $r=3$ and $t=2$ into Corollary \ref{runifbergek2t} (a), we get a lower bound of $n^{3/2}/3\sqrt{3}$, which is larger. For this particular case, the best upper bound known is $ex_3(n,K_{2,2})\le (1+o(1))n^{3/2}/\sqrt{10}$ due to Ergemlidze, Gy\H{o}ri, Methuku, Tompkins and Salia \cite{C_4}.
\end{rem}

\subsubsection{Improved bounds for Berge-$C_{2k}$}

\begin{kov}
\label{even_cycle}
Let $k\ge 5$. Then $$ex_3(n,C_{2k})\le \frac{2k-3}{3} \cdot ex(n,C_{2k}).$$
\end{kov}

\noindent
Note that $ex(n, K_3, C_{2k})\le \frac{2k-3}{3}ex(n, C_{2k})$ (the second inequality from Theorem \ref{foc2k}) and Theorem \ref{gp2} implies $$ex_3(n, C_{2k})\le \frac{2k}{3}ex(n, C_{2k}),$$
which is the first inequality from Theorem \ref{foc2k}. Here we remove the difference between $ex_3(n,C_{2k})$ and $ex(n,K_3,C_{2k})$, as we do in the case of $K_{2,t}$ in Corollary \ref{maincor}.

\vspace{2mm}

\noindent
For the $r$-uniform case, we have the following corollary giving a new proof of Theorem \ref{jiangma}. We note that the multiplicative factor given in Corollary \ref{jiangma2} is better than the one obtained in the proof of Theorem \ref{jiangma} by Jiang and Ma \cite{JM2016}, whenever $r \ge 8$ or $k \ge 3$.

\begin{kov}\label{jiangma2} If $n,k \ge 2$ and $r \ge 4$, then we have
$$ex_r(n,C_{2k}) \le \max\left \{\frac{1}{r(k-1)}\binom{2k-2}{r-1}, g(r) \right \}  \cdot ex(n, C_{2k}) =$$ $$= O_r(k^{r-2}) \cdot ex(n,C_{2k}).$$
\end{kov}
\begin{proof}
Using Theorem \ref{Luo}, we get $$ex(n, K_{r-1}, P_{2k-1}) = \frac{n}{2k-2} \binom{2k-2}{r-1}.$$
So in Theorem \ref{main}, we can choose $i = 1$, $F' = P_{2k-1}$ and $c = \frac{1}{2k-2}\binom{2k-2}{r-1}$. To decide whether part (a) or part (b) of Theorem \ref{main} applies, we need to compare $$c = \frac{1}{2k-2}\binom{2k-2}{r-1} \ \textrm{ and } \ \frac{rg(r)}{2}.$$ If the first one is larger, then we get $$ex_r(n,C_{2k}) \le \frac{2}{r(2k-2)}\binom{2k-2}{r-1} \cdot ex(n, C_{2k}).$$ If the second one is larger, we get $$ex_r(n,C_{2k}) \le g(r) \cdot ex(n, C_{2k}).$$ \end{proof}

\begin{rem}
	Let us assume $2k>\lceil r/2\rceil$, take a bipartite graph $G$ of girth more than $2k$ containing $\lfloor n/r\rfloor$ vertices in both color classes and replace each vertex in one part by $\lfloor r/2\rfloor$ copies of it, and each vertex in the other part by $\lceil r/2\rceil$ copies of it. We claim that if the resulting $r$-uniform hypergraph $\cH$ contains a Berge-$C_{2k}$, then $G$ contains a cycle of length at most $2k$, which is impossible. Indeed, consider a Berge cycle of length $2k$, $v_1$,$h_{1}$,$v_2$,$h_{2},\ldots,v_{2k}$,$h_{2k}$,$v_1$ in $\cH$. Then each $v_i$ is a copy of a vertex $u_i$ of $G$. So this Berge cycle corresponds to a closed walk $u_1u_2,\ldots,u_{2k}u_1$ in $G$ of length $2k$. We claim that this closed walk contains a cycle of length at most $2k$. Indeed, otherwise either an edge is repeated in the walk or it consists of only one vertex; we will show both of these cases are impossible: As $2k>\lceil r/2\rceil$, there are two vertices $v_i$ and $v_j$ that are copies of two different vertices $u_i$ and $u_j$ (respectively) of $G$, which means the walk contains at least two different vertices of $G$. Also observe that if an edge is repeated in the closed walk $u_1u_2,\ldots,u_{2k}u_1$, say $u_iu_{i+1} = u_ju_{j+1}$ (addition in the subscripts is modulo $2k$), then we must have $h_i = h_j$ contradicting the definition of a Berge cycle.    
  This gives a lower bound of $$ex_r(n, C_{2k}) \ge ex\left(\left\lfloor\frac{n}{r}\right\rfloor, \left\lfloor\frac{n}{r}\right\rfloor, \{C_3,C_4, \ldots, C_{2k} \}\right).$$
\end{rem}

\subsection{Asymptotics for Berge-$K_{2,t}$ - the linear case} 
\label{linear_case}

First we prove the following upper bound.

\begin{thm}\label{linearberge}
For all $r, t \ge 2$, we have $$ex_r(n, \{C_2, K_{2,t}\}) \le \frac{\sqrt{t-1}}{r(r-1)}n^{3/2} + O(n).$$ 
\end{thm}

\noindent
Note that putting $r=2$ in Theorem \ref{linearberge}, we can recover the upper bound in F\"uredi's theorem - Theorem \ref{K2tgraph}.


\noindent

\noindent

\vspace{1mm}

Our main focus in the rest of this section is to prove lower bounds and determine the asymptotics of the Tur\'an number of Berge-$K_{2,t}$ in  $3$-uniform  linear hypergraphs. Putting $r=3$ in Theorem \ref{linearberge} we get an upper bound of $$(1+o(1))\frac{\sqrt{t-1}}{6}n^{3/2},$$ and putting $r=3$ in Theorem \ref{Timmons_Palmer}  we get a lower bound of 

$$\frac{\sqrt{t-1}}{6\sqrt{3}}n^{3/2} + O(n)$$ for some special $n$ and $t$. First we present a slightly weaker lower bound but its proof is much simpler than that of Theorem \ref{Timmons_Palmer} and it also works for every $n$ and $t$:

\begin{thm}\label{linearlower}
If $n \ge t \ge 2,$ we have 
$$ex_3(n, \{C_2, K_{2,t}\}) \ge \frac{\sqrt{t-1}}{12}n^{3/2} + O(n).$$
\end{thm}

\noindent
However, when $t$ is large enough, we can do much better. In Theorem \ref{linearlower2} below, we prove a lower bound of $$(1+o_t(1))\frac{\sqrt{t-1}}{6}n^{3/2},$$ where $o_t$ depends on $t$. Thus, it shows that for large enough $t$ our upper bound in Theorem \ref{linearberge} is close to being asymptotically correct for $r=3$.

\begin{thm}\label{linearlower2} 
There is an absolute constant $c$ such that for any $t \ge 2$, we have,
$$ex_3(n, \{C_2, K_{2,t}\}) \ge \left(1-\frac{c}{\sqrt{t-1}}\ln^{3/2}(t-1)\right)\frac{\sqrt{t-1}}{6}n^{3/2}.$$
\end{thm}

Note that for $t=2$, this gives a lower bound matching the upper bound in Theorem \ref{linearberge} for $r=3$. Therefore, we can recover a sharp result of Ergemlidze, Gy\H{o}ri and Methuku in \cite{BEM}, showing $ex_3(n, \{C_2,C_4\}) = (1+o(1))n^{3/2}/6$.  

\section{Proofs of the results about $r$-uniform Berge-$F$-free hypergraphs}
\label{BergeF}

\subsection{Proof of Theorem \ref{main}}




Let us introduce some notation. Let $\cH$ be an $r$-uniform Berge-$F$-free hypergraph. We call a pair of vertices $u,v$ a \textit{blue edge} if it is contained in at least one and at most $\binom{r}{2}-2$ hyperedges in $\cH$ and a \textit{red edge} if it is contained in more than $\binom{r}{2}-2$ hyperedges. 



We call a hyperedge \textit{blue} if it contains a blue edge, and \textit{red} otherwise. 
Let us denote the set of blue edges by $S$ and the number of blue hyperedges by $s$.
We choose a largest subset $S'\subset S$ with the property that every blue hyperedge contains at most one edge of $S'$.  

\begin{clm}\label{blue} $|S'|\ge s/f(r)$.

\end{clm}

\begin{proof} 
We build an auxiliary bipartite graph $A$ with parts $P$ and $Q$ where $P$ consists of all the blue edges and $Q$ consists of all the blue hyperedges of $\cH$, and we connect a vertex of $P$ with a vertex of $Q$ if the corresponding blue edge is contained in the corresponding blue hyperedge. 

By definition, $S'$ is the largest subset of $P$ such that any two vertices of $S'$ are at distance more than two in the graph $A$. We claim that every vertex of $Q$ is at distance at most three from a vertex of $S'$. Indeed, otherwise any of its neighbors can be added to $S'$.

Now we show that the number of vertices in $Q$ that are at distance at most three from a vertex of $S'$ is at most $|S'|f(r)$. Indeed, a blue edge is contained in at most $(\binom{r}{2}-2)$ blue hyperedges, and they each contain at most $\binom{r}{2}-1$ other blue edges;  those blue edges are in turn contained in at most $\binom{r}{2}-3$ other blue hyperedges. So a vertex of $S'$ is at distance at most three from at most $(\binom{r}{2}-2)(1+(\binom{r}{2}-1)(\binom{r}{2}-3)) = f(r)$ vertices in $Q$. Since we must have $|Q| = s \le |S'|f(r)$, the proof is complete.


\end{proof}

Now consider an auxiliary graph $G$, consisting only of the blue edges of $S'$ and all the red edges.
Let us assume there is a copy of $F$ in $G$. We build an auxiliary bipartite graph $B$. One of its classes $B_1$ consists of the edges of that copy of $F$, and the other class $B_2$ consists of the hyperedges of $\cH$ that contain them. We connect a vertex of $B_1$ with a vertex of $B_2$ if the corresponding hyperedge of $\cH$ contains the corresponding edge of $F$. Note that every hyperedge can contain at most $\binom{r}{2}-1$ edges from a copy of $F$ (since $F$ is $K_r$-free), thus vertices of $B_2$ have degree at most $\binom{r}{2}-1$.

Notice that a matching in $B$ covering $B_1$ would give a Berge-$F$ in $\cH$. Thus by Hall's theorem there is a subset $X\subset B_1 = E(F)$ with $|N(X)|< |X|$ for any copy of $F$ in $G$. We call such a subset $X$ \textit{bad}. 

\begin{clm}\label{redg2} There is a blue edge in every copy of $F$ in $G$.
\end{clm}
\begin{proof} Otherwise we can find a bad set $X\subset B_1$ such that every element in it is a red edge of $G$, thus the corresponding vertices have degree at least $\binom{r}{2}-1$ in $B$. On the other hand since every vertex of $N(X)$ is in $B_2$, they have degree at most $\binom{r}{2}-1$ in $B$. This implies $|N(X)| \ge |X|$, contradicting the assumption that $X$ is a bad set.
\end{proof}

The following claim shows that every bad subset of edges in a copy of $F$ in $G$ contains a red edge which is contained in few red hyperedges. Our plan will be to recolor such a red edge of a bad set in each copy of $F$ in $G$ to green, to make sure that every copy of $F$ in $G$ contains a green edge.
%
%
%
\begin{clm}\label{clmbad} Every bad set $X$ (in any copy of $F$ in $G$) contains a red edge that is contained in at most $\binom{r}{2}-2$ red hyperedges.

\end{clm}

\begin{proof} Let us assume indirectly that every red edge of $G$ is contained in at least $\binom{r}{2}-1$ red hyperedges. Let $x$ be the number of red hyperedges in $N(X)$ and $y$ be the number of blue hyperedges. Then the number of blue edges in $X$ is at most $y$ (since $S'$ has the property that every blue hyperedge contains at most one pair of $S'$). Since $|N(X)| < |X|$, this implies the number of red edges in $X$ is more than $x$, hence the number of edges in $B$ between the red edges in $X$ and red hyperedges is at least $(x + 1)(\binom{r}{2}-1)$. However, a hyperedge in $B_2$ has at most $\binom{r}{2}-1$ neighbors in $B_1$ and so there can be at most $x(\binom{r}{2}-1)$ edges in $B$ between red hyperedges and red edges in $X$, a contradiction.
\end{proof}

We will call a red edge (in each copy of $F$ in $G$) guaranteed by the above claim, \emph{special red edge}. As every copy of $F$ in $G$ contains a bad set, it contains a special red edge too.

We consider each copy of $F$, one by one, and recolor a special red edge in it to \textit{green}, and also recolor all the red hyperedges containing it to \textit{green}. Note that it is possible that some of the red hyperedges containing a special red edge of $G$ may have already turned green. However, after this procedure, the total number of green hyperedges of $\cH$ is obviously at most $\binom{r}{2}-2$ times the number of green edges of $G$. Notice that each remaining red hyperedge still contains $\binom{r}{2}$ red edges of $G$.

Let us now recolor the remaining red edges of $G$ and red hyperedges of $\cH$ to \textit{purple} to avoid confusion. Thus $G$ now contains blue, green and purple edges, while $\cH$ contains blue, green and purple hyperedges. Every blue hyperedge of $\cH$ contains at most one blue edge of $G$, and every green hyperedge of $\cH$ contains a green edge of $G$ (possibly more than one), while a purple hyperedge contains $\binom{r}{2}$ purple edges of $G$ (i.e., every pair contained in a purple hyperedge is a purple edge of $G$). 

Furthermore, let $G_1$ be the subgraph of $G$ consisting of blue and purple edges, and let $G_2$ be the subgraph of $G$ consisting of green and purple edges. Clearly $G_1$ is $F$-free because every copy of $F$ in $G$ contains a green edge. We claim $G_2$ is also $F$-free -- indeed, notice that we recolored only red edges to green or purple, so the edges in $G_2$ were all originally red. Therefore, by Claim \ref{redg2}, $G_2$ cannot contain a copy of $F$.

\begin{clm}
\label{based_on_x}
If an $F$-free graph $G$ contains $x$ edges, then it contains at most $$\min\left\{\frac{2cxn^{i-1}}{r}, cx(r-1) \left(\frac{2ex(n,F)}{n}\right)^{i-1}\right\}$$ copies of $K_r$.

\end{clm}

\begin{proof} Obviously the neighborhood of every vertex is $F'$-free. An $F'$-free graph on $d(v)$ vertices contains at most $$ex(d(v),K_{r-1},F')\le cd(v)^{i}$$ copies of $K_{r-1}$ by the definition of $c$. It means $v$ is in at most that many copies of $K_r$, so summing up for every vertex $v$, every $K_r$ is counted $r$ times.
On the other hand as $\sum_{v\in V(G)} d(v)=2x$ and $d(v)\le n$ we have $$\sum_{v\in V(G)} cd(v)^{i} \le \sum_{v\in V(G)} cn^{i-1} d(v) = 2cxn^{i-1},$$ 
showing that the number of copies of $K_r$ in $G$ is at most $$\frac{2cxn^{i-1}}{r}.$$

\noindent
Now we show that the number of copies of $K_r$ is also at most $$cx(r-1) \left( \frac{2ex(n,F)}{n} \right)^{i-1}.$$ Let $a$ be the number of the copies of $K_r$ in $G$. Let us consider an edge that is contained in less than $a/x$ copies of $K_r$, and delete it. We repeat this as long as there exists such an edge. Altogether we deleted at most $x$ edges, hence we deleted less than $a$ copies of $K_r$. So the resulting graph $G_1$ is non-empty. Let us delete the isolated vertices of $G_1$. The resulting graph $G_2$ is $F$-free on, say, $n' < n$ vertices, hence it contains at most $ex(n',F)$ edges. This implies $G_2$ contains a vertex $v$ with degree $$d(v)\le \frac{2ex(n',F)}{n'}\le \frac{2ex(n,F)}{n}.$$ 

\noindent
Let us consider the number of copies of $K_{r-1}$ in the neighborhood of $v$ in $G_2$. On one hand it is at most $$ex(d(v),K_{r-1},F')\le cd(v)^{i}.$$ On the other hand, it is equal to the number of copies of $K_r$ that contain $v$, which is at least $$\frac{ad(v)}{x(r-1)}.$$ Indeed, the $d(v)$ edges incident to $v$ are all contained in at least $a/x$ copies of $K_r$, and such copies are counted $r-1$ times. So combining, we get $$a\le cx(r-1)d(v)^{i-1}\le cx(r-1) \left(\frac{2ex(n,F)}{n}\right)^{i-1},$$ completing the proof of the claim.
\end{proof}

Now we continue the proof of Theorem \ref{main}. Let $x$ be the number of purple edges in $G$.
Then, by Claim \ref{based_on_x}, the number of copies of $K_r$ consisting of purple edges in $G$ is at most $$\frac{2cxn^{i-1}}{r},$$ but then the number of purple hyperedges in $\cH$ is also at most this number because any pair contained in a purple hyperedge forms a purple edge in $G$.

Let $y:=ex(n,F)$. Then the number of blue edges in $G$ is at most $y-x$ as $G_1$ is $F$-free, and similarly, the number of green edges in $G$ is at most $y-x$. 

By Claim \ref{blue} the total number of blue hyperedges in $\cH$ is at most $f(r)$ times the number of blue edges, i.e., at most $f(r)(y-x).$

Moreover, we claim that the number of green hyperedges in $\cH$ is at most  $\binom{r}{2}-2$ times the number of green edges of $G$ -- indeed, by Claim \ref{clmbad}, any (special) red edge of $G$ that was recolored green, was originally contained in at most $\binom{r}{2}-2$ red hyperedges (which were all recolored to become green hyperedges). Thus a green edge of $G$ is in at most $\binom{r}{2}-2$ green hyperedges of $\cH$ and every green hyperedge contains a green edge. Therefore, the number of green hyperedges in $\cH$ is at most $\left(\binom{r}{2}-2\right)(y-x).$

Therefore the total number of hyperedges in $\cH$ is at most 

$$\frac{2cxn^{i-1}}{r}+ \left(f(r)+\binom{r}{2}-2\right)(y-x)\le\max\left\{\frac{2cn^{i-1}}{r},g(r)\right\}(x+y-x),$$

\vspace{2mm}

\noindent
and we are done with (a) and (b) by the assumption on $c$. 

For (c) we use the other upper bound on the number of $K_r$'s from Claim \ref{based_on_x}, namely: $$cx(r-1) \left(\frac{2ex(n,F)}{n}\right)^{i-1}.$$ Observe that it goes to infinity as $n$ grows, since $i>1$.
The same calculation gives that for large enough $n$ the number of hyperedges is at most $$\max\left\{c(r-1) \left(\frac{2ex(n,F)}{n}\right)^{i-1},g(r)\right\}(x+y-x)=c(r-1) \left(\frac{2ex(n,F)}{n}\right)^{i-1}y$$ $$\le c(r-1)ex(n,F)^{i} \left(\frac{2}{n}\right)^{i-1}.$$

\qed

\subsection{Proof of Corollary \ref{runifbergek2t}}
First we prove (a):

\vspace{2mm}

Let $G$ be a bipartite $K_{2,t}$-free graph with $\frac{n}{r}$ vertices in both color classes and containing $$\sqrt{t-1} \left(\frac{n}{r}\right)^{3/2} (1+o(1))$$ edges. The existence of such a graph is guaranteed by Theorem \ref{biparitieturanK2tgraph}. Let $$A = \{a_1, a_2, \ldots, a_{\frac{n}{r}} \}, \textrm{ and } B = \{b_1, b_2, \ldots, b_{\frac{n}{r}} \}$$ be the color classes of $G$.

Let us replace each vertex $a_i \in A$ with a set $A_i$ of $\lf \frac{r}{2} \rf$ copies of $a_i$, and each vertex $b_i \in B$ with a set $B_i$ of $\lc \frac{r}{2} \rc$ copies of $b_i$ to get an $r$-uniform hypergraph $\mathcal{H}$. Let $$A_{new} := \cup_i A_i, \textrm{ and}$$ $$ B_{new} := \cup_i B_i.$$ It is easy to see that the number of hyperedges in $\mathcal{H}$ is equal the number of edges in $G$, as required. It remains to show that $\mathcal{H}$ is Berge-$K_{2,t}$-free.
	
	Suppose for a contradiction that $\mathcal{H}$ contains a Berge-$K_{2,t}$. Then there is a bijective map from the hyperedges of the Berge-$K_{2,t}$ to the edges of the graph $K_{2,t}$ such that each edge is contained in the hyperedge that was mapped to it. Let $\{p,q\}$ and $T$ be the color classes of $K_{2,t}$.
	If $\{p,q\} \subset A_{new}$ and $T \subset B_{new}$ (or vice versa), then $p$ and $q$ cannot be in the same $A_i$ as each $A_i$ and each vertex of $B_{new}$ are contained in at most one hyperedge of $\mathcal{H}$, however the hyperedges of the Berge-$K_{2,t}$ containing the edges $pr, qr$ for some $r \in T$ must be different, a contradiction. Therefore, $p$ and $q$ belong to distinct $A_i$ and  similarly, the vertices of  $T$ must belong to distinct  $B_i$, but this implies that $G$ contains a $K_{2,t}$, a contradiction. So there are two vertices $x \in \{p,q\}$ and $y \in T$ such that $x, y \in A_{new}$ or $x, y \in B_{new}$.
    
    Suppose first that $x, y \in B_{new}$. There must be a hyperedge in $\mathcal{H}$ containing both $x$ and $y$.  However, there is no  hyperedge in $\mathcal{H}$ containing a vertex of $B_i$ and a vertex of $B_j$ with  $i \not =  j$, so $x$ and $y$ are both contained in  some $B_i$.  Every vertex of $\{p,q\} \cup T$ must be contained in a hyperedge with $x$ or $y$, thus each vertex of $\{p,q\} \cup T$ must be in $B_i$ or $A_{new}$. As the size of $B_i$ is $$\lc \frac{r}{2} \rc < \abs{\{p,q\} \cup T} = t+2$$ by assumption, there must be at least one vertex $z \in \{p,q\} \cup T$ in $A_{new}$. There is exactly one hyperedge of $\mathcal{H}$ that contains $z$ and any other vertex of $B_i \cup A_{new}$. However, the degree of $z$ in the Berge-$K_{2,t}$ is at least $2$, a contradiction.
    
    If $x, y \in A_{new}$ then we can again get a contradiction by the same reasoning as above. Therefore, $\mathcal{H}$ is Berge-$K_{2,t}$-free. 

\vspace{5mm}

\noindent
Now we prove (b):

\vspace{2mm}

\noindent 



\noindent
In a $K_{1,t}$-free graph, since the degree of any vertex is at most $t-1$, there are at most $$\binom{t-1}{r-2}$$ cliques of size $r-1$ containing any vertex. Thus we get the following.
$$ex(n, K_{r-1}, K_{1,t}) \le \frac{n}{r-1}\binom{t-1}{r-2} = \frac{n}{t} \binom{t}{r-1}.$$

\noindent
Therefore in Theorem \ref{main}, we can choose $i = 1$, $F' = K_{1,t}$ and $c = \frac{1}{t}\binom{t}{r-1}$. To apply Theorem \ref{main} part (a) or (b), we need to compare $$c = \frac{1}{t}\binom{t}{r-1} \ \textrm{ and } \ \frac{rg(r)}{2}.$$ If the first one is larger, then we get $$ex_r(n,K_{2,t}) \le \frac{2}{r \cdot t}\binom{t}{r-1} \cdot ex(n, K_{2,t}),$$ and by Theorem \ref{K2tgraph} we are done.  If the second one is larger, we get $$ex_r(n,K_{2,t}) \le g(r) \cdot ex(n, K_{2,t}),$$
and we are again done by Theorem \ref{K2tgraph}.\qed


\section{Proofs of the results about $r$-uniform linear Berge-$K_{2,t}$-free hypergraphs}
\label{linear}
\vspace{5mm}

\subsection{Proof of Theorem \ref{linearberge}}

Let $\cH$ be an $r$-uniform linear hypergraph containing no Berge-$K_{2,t}$. 

\vspace{1mm}

First let us fix $v \in V(\cH)$. Let the first neighborhood and second neighborhood of $v$ in $\cH$ be defined as $$N^{\cH}_1(v) := \{x \in V(\cH) \setminus \{v \}  \mid \exists h \in E(\cH)\text{ such that } v,x \in h \}, \textrm{ and}$$ $$N^{\cH}_2(v) := \{x \in V(\cH) \setminus (N^{\cH}_1(v) \cup \{v\}) \mid \exists h \in E(\cH) \text{ such that } x \in h \text{ and } h \cap N^{\cH}_1(v) \not = \emptyset\},$$ respectively.

\begin{clm}
	\label{small_degreeG1}
For any $u \in N^{\cH}_1(v)$,  the number of hyperedges $h \in E(\cH)$ containing $u$ such that $$\abs{h \cap N^{\cH}_1(v)} \ge 2$$ is at most $(r-1)(t-1)+1 \ (\le (r-1)t)$. 
\end{clm}


\begin{proof}
Suppose for a contradiction that there is a vertex $u \in N^{\cH}_1(v)$ which is contained in $(r-1)(t-1)+2$ hyperedges $h$ such that $\abs{h \cap N^{\cH}_1(v)} \ge 2$. At most one of them contains $v$ because $\cH$ is linear. 

From each of the $(r-1)(t-1)+1$ other hyperedges $h_i$ ($1 \le i \le (r-1)(t-1)+1$), that do not contain $v$, we select exactly one pair $uy_i \subset h_i$ arbitrarily. These pairs are distinct since $H$ is linear. Then (by pigeonhole principle) it is easy to see that there exist $t$ distinct vertices $$p_1, p_2, \ldots, p_t \in \{y_1, y_2, \ldots, y_{(r-1)(t-1)+1}\}$$ and $t$ distinct hyperedges $f_1, f_2, \ldots, f_t$ containing $v$ such that  $p_i \in f_i$. The  $t$ hyperedges containing  the pairs $up_i$ and the $t$ hyperedges $f_i$ ($1 \le i \le t$), form a  Berge-$K_{2,t}$ in $\cH$, a contradiction. 
\end{proof}

For each $u \in N^{\cH}_1(v)$, let $$E_u := \{h \in E(\cH)  \mid h \cap N^{\cH}_1(v) = \{u\} \}, \textrm{ and}$$ $$V_u := \{w \in N^{\cH}_2(v) \mid \exists h \in E_u \text{ with } w \in h \}.$$ Notice that by Claim \ref{small_degreeG1} we have $$\abs{E_u} \ge d(u)-(r-1)t \ \textrm{ and } \ \abs{V_u} = (r-1) \abs{E_u}$$ (except for $r=2$, in which case $|E_u| \ge d(u)-t+1$ and $|V_u|=|E_u|-1$, because the edge $vu \in E_u$. However, inequality \eqref{Vu} will still hold).

\vspace{2mm}

\noindent
Therefore,
\begin{equation}
\label{Vu}
\abs{V_u} \ge (r-1)d(u)-(r-1)^2 t.
\end{equation}

Let the hyperedges incident to $v$ be  $e^v_1, e^v_2, \ldots, e^v_{d(v)}$, where 
$$e^v_i =:  \{v,u_{1,i},u_{2,i}, \ldots, u_{r-1,i}\}$$ for $1 \le i \le d(v)$,
and let us define the sets  $$V_i :=  \cup_{j=1}^{r-1} V_{u_{j,i}}$$ for each $ 1 \le i \le d(v)$.

\vspace{2mm}

\begin{clm}
\label{union_set}
For each $1 \le i \le d(v)$, we have $$\abs{V_i} \ge \sum_{j =1}^{r-1}\abs{V_{u_{j,i}}} - \binom{r-1}{2}(2rt).$$
\end{clm}

\begin{proof} 

Note that
\begin{equation}
\label{inclusion_exclusion}
\abs{V_i} =  \abs{\cup_{j=1}^{r-1} V_{u_{j,i}}} \ge \sum_{j =1}^{r-1}\abs{V_{u_{j,i}}} - \sum_{1 \le p < q \le r-1}{\abs{V_{u_{p,i}}  \cap V_{u_{q,i}}}}.
\end{equation}

\noindent
First we will show that 

$$\abs{V_{u_{p,i}}  \cap V_{u_{q,i}}} \le (2r-3)t-1.$$

\noindent
Suppose by contradiction that  $\abs{V_{u_{p,i}}  \cap V_{u_{q,i}}} \ge (2r-3)t$. We will construct an auxiliary graph $G$ whose vertex set is $V_{u_{p,i}}  \cap V_{u_{q,i}}$ and whose edge set is the union of the two sets $$\{xy \mid xy \subset h \text{ for some } h \in E_{u_{p,i}} \text{ and } x, y \in  V_{u_{p,i}}  \cap V_{u_{q,i}}\}, \ \text{and}$$ $$\{xy \mid xy \subset h \text{ for some } h \in E_{u_{q,i}} \text{ and } x, y \in  V_{u_{p,i}}  \cap V_{u_{q,i}}\}.$$ 

\noindent
It is easy to see that each set consists of pairwise vertex disjoint cliques of size at most $r-1$. Therefore, the maximum degree in $G$ is at most $2(r-2)$, so 
it has chromatic number at most $2r-3$, which implies that it has an independent set $I$ of size at least 

$$\frac{\abs{V(G)}}{2r-3} = \frac{\abs{V_{u_{p,i}}  \cap V_{u_{q,i}}}}{2r-3} \ge t.$$ 

\vspace{1mm}

\noindent
Let  $x_1, x_2, \ldots, x_{t} \in I$ be  distinct vertices. Then consider the set of hyperedges containing the pairs  $$u_{p,i}x_1, u_{p,i}x_2, \ldots, u_{p,i}x_{t}$$ 

\noindent
and the set of hyperedges containing the pairs 
$$u_{q,i}x_1, u_{q,i}x_2, \ldots, u_{q,i}x_{t}.$$ 

\noindent
The hyperedges in the first set are different from each other since $x_1, x_2, \ldots, x_{t}$ is an independent set. Similarly, the hyperedges in the second set are different from each other. A hyperedge of the first set and a hyperedge of the second set can not be same since that would imply that there is a hyperedge in $E_{u_{p,i}}$ containing the pair $u_{p,i}u_{q,i}$. However, this is impossible since such a hyperedge contains exactly one vertex from  $N^{\cH}_1(v)$ by definition. So all the hyperedges are different and they form a Berge-$K_{2,t}$, a contradiction. 
Therefore, we have $$\abs{V_{u_{p,i}}  \cap V_{u_{q,i}}} \le (2r-3)t-1 < 2rt.$$ Using this upper bound in \eqref{inclusion_exclusion}, we  get

$$
\abs{V_i} \ge \sum_{j =1}^{r-1}\abs{V_{u_{j,i}}} - \sum_{1 \le p < q \le r-1}{\abs{V_{u_{p,i}}  \cap V_{u_{q,i}}}} \ge \sum_{j =1}^{r-1}\abs{V_{u_{j,i}}} - \binom{r-1}{2}(2rt),
$$
completing the proof of the claim.
\end{proof}

\begin{clm}
\label{sumofUnionsets} We have
$$\sum_{i=1}^{d(v)} \abs{V_i} \le (t-1)n.$$
\end{clm}
\begin{proof}
It suffices to show that a vertex  $x \in V(\cH)$ belongs to at most  $t-1$ of the sets $V_i$ for any $1 \le i \le d(v)$. Suppose for a contradiction that there is a vertex $x$ that is contained in $t$ sets  $V_{i_1}, V_{i_2}, \ldots, V_{i_t}$ for some distinct $i_1, i_2, \ldots, i_t \in \{1,2,\ldots,d(v)\}$. For notational simplicity, we may assume that $i_1 = 1$, $i_2 = 2$, . . . , and $i_t = t$. This means that there are $t$ hyperedges $h_1, h_2, \ldots, h_t$ containing the pairs $xz_1, xz_2, \ldots, xz_t$, respectively, where  $$z_j \in e^v_j \setminus \{v\} = \{u_{1,j},u_{2,j}, \ldots, u_{r-1,j}\}$$ for $1 \le j \le t$. 
The hyperedges $h_1, h_2, \ldots, h_t$ are distinct since they contain exactly one vertex from $N^{\cH}_1(v)$. Moreover, the $t$ hyperedges $e^v_j$ for  $1 \le j \le t$ are distinct from  $h_1, h_2, \ldots, h_t$ as a hyperedge in the former set contains $v$ but not $x$ and a hyperedge in the latter set contains $x$ but not $v$.  Therefore, these $2t$ hyperedges form a Berge-$K_{2,t}$, a contradiction.
\end{proof}

\begin{lemma}
\label{upperboundlinear}
For any $v \in V(\cH)$, we have $$\sum_{u \in N^{\cH}_1(v)}  (r-1)d(u) \le (t-1)n + (r-1)(2r^2-4r+1)t d(v).$$
\end{lemma}
\begin{proof}
By \eqref{Vu}, 
\begin{equation}\label{1}
\begin{split}
\sum_{u \in N^{\cH}_1(v)}  (r-1)d(u) \le \sum_{u \in N^{\cH}_1(v)} (\abs{V_u} + (r-1)^2 t) = \\
(r-1)^2 t \cdot (r-1)d(v)+\sum_{u \in N^{\cH}_1(v)}\abs{V_u}.
\end{split}
\end{equation}
Moreover, by Claim \ref{union_set}, $$\sum_{u \in N^{\cH}_1(v)}\abs{V_u} =  \sum_{i=1}^{d(v)}\sum_{j =1}^{r-1}\abs{V_{u_{j,i}}} \le \sum_{i=1}^{d(v)}(\abs{V_i} + \binom{r-1}{2}(2rt)),$$ and so using Claim \ref{sumofUnionsets} 
we have 
\begin{equation}
\label{2}
\sum_{u \in N^{\cH}_1(v)}\abs{V_u} \le (t-1)n + \binom{r-1}{2}(2rt)d(v).
\end{equation}
Combining \eqref{1} and \eqref{2}, we get  $$\sum_{u \in N^{\cH}_1(v)}  (r-1)d(u) \le (t-1)n +  \left(\binom{r-1}{2}(2rt) + (r-1)^2 t \cdot (r-1)\right)d(v).$$
Simplifying, we get

$$\sum_{u \in N^{\cH}_1(v)}  (r-1)d(u) \le (t-1)n + (r-1)(2r^2-4r+1)t d(v),$$ as desired.
\end{proof}

\noindent
On the one hand, if $d$ denotes the average degree in $\cH$, by Lemma \ref{upperboundlinear} we have $$\sum_{v \in V(H)} \sum_{u \in N^{\cH}_1(v)} (r-1)d(u) \le  (t-1)n^2 + (r-1)(2r^2-4r+1)tnd.$$ 
On the other hand, 
$$\sum_{v \in V(\cH)} \sum_{u \in N^{\cH}_1(v)} (r-1)d(u) = \sum_{u \in V(\cH)} (r-1)d(u) \cdot (r-1)d(u) =$$ $$\sum_{u \in V(\cH)} (r-1)^2d(u)^2 \ge (r-1)^2nd^2.$$ Here the first equality follows by taking an arbitrary vertex $u$ and counting how many times it appears in the first neighborhood of a vertex $v$, while the last inequality follows from the Cauchy-Schwarz inequality.  
So combining, we get  $$(r-1)^2nd^2 \le (t-1)n^2 + (r-1)(2r^2-4r+1)tnd.$$  That is, $$(r-1)^2d^2 \le (t-1)n + (r-1)(2r^2-4r+1)td.$$ Rearranging, we get $$((r-1)d-c_1(r,t))^2  \le (t-1)n + c_2(r,t),$$ where $$c_1(r,t) := \frac{(2r^2-4r+1)t}{2} \ \textrm{ and } \ c_2(r,t) := (c_1(r,t))^2.$$
This  gives that $$d  \le \frac{\sqrt{n(t-1)}}{r-1} \left(1+ O\left(\frac{1}{n}\right)\right) + c_3(r,t)$$ for some constant $c_3(r,t)$ depending only on $r$ and $t$.
So the number of hyperedges in $\cH$ is $$\frac{nd}{r} \le \frac{n^{3/2}\sqrt{t-1}}{r(r-1)} + O(n),$$ completing the proof of our  theorem.
\qed

\subsection{Proofs of Theorem \ref{linearlower} and \ref{linearlower2}}

In order to prove both theorems, we take the $K_{2,t}$-free graph $G$ constructed by F\"uredi \cite{F1996} (which is used to prove the lower bound in Theorem \ref{K2tgraph}), and replace its triangles by hyperedges as usual. However, the resulting hypergraph is far from linear, so our main idea is to delete some hyperedges in it to get a linear hypergraph. The graph $G$ contains many triangles and this number is calculated by Alon and Shikhelman to prove their lower bound in Theorem \ref{ask2t}. In our proofs of both theorems (Theorem \ref{linearlower} and \ref{linearlower2}) we do not need many specific properties of $G$. In the proof of Theorem \ref{linearlower} we use that it is $K_{2,t}$-free and contains $$(1+o(1))\frac{1}{6}(t-1)^{3/2}n^{3/2}$$ triangles. In the proof of Theorem \ref{linearlower2} we also use that it contains  $$(1+o(1))\frac{1}{2}(t-1)^{1/2}n^{3/2}$$ edges and all but $o(n^{3/2})$ edges are contained in $t-1$ triangles, while the remaining edges are contained in $t-2$ triangles. One can easily check these well-known properties of F\"uredi's construction \cite{F1996}, so we omit the proofs of these properties.

\vspace{4mm}

To conclude the proof of Theorem \ref{linearlower}, we construct an auxiliary graph $G'$. Its vertices are the triangles of $G$, and two vertices of $G'$ are connected by an edge if the corresponding triangles in $G$ share an edge. Obviously, we want to find a large independent set in $G'$. A theorem of Fajtlowicz  states the following.

\begin{thm}[\cite{F1978}]
\label{Fajtlow}

Any graph $F $ contains an independent set of size at least $$\frac{2|V(F)|}{\Delta(F)+\omega(F)+1},$$ where $\Delta(F)$ and $\omega(F)$ denotes the maximal degree and the size of the maximal clique of $F$, respectively.

\end{thm}

\noindent
Clearly we have $\Delta(G')\le 3(t-2) = 3t-6$ since each of the three edges of a triangle in $G$ is contained in at most $t-2$ other triangles. 
Now notice that if a set of triangles of $G$ pairwise intersect in two vertices then they either share a common edge or they are all contained in a $K_4$. In both cases, it is easy to see that $\omega(G') \le t+1$. Substituting these bounds in Theorem \ref{Fajtlow} and using that $|V(G')| = (1+o(1))\frac{1}{6}(t-1)^{3/2}n^{3/2}$ completes the proof of Theorem \ref{linearlower}.

\qed
\vspace{4mm}

To prove Theorem \ref{linearlower2}, we define an auxiliary hypergraph $\cH$ to be the $3$-uniform
hypergraph whose vertices are the edges of $G$, and three vertices $e_1$, $e_2$ and $e_3$ form a hyperedge in $\cH$ if there is a triangle in $G$ whose edges are $e_1$, $e_2$ and $e_3$. Then $\cH$ is linear since given any two edges of $G$, there is at most one triangle in $G$ that
contains both of them. Further, $\cH$ is 3-uniform and all but $o(n^{3/2})$ vertices in $\cH$ have degree $t-1$, while the rest have degree $t-2$. 
It is easy to see that we can construct another hypergraph $\cH'$ by adding a set $X$ of $o(n^{3/2})$ vertices to the vertex set of $\cH$, such that $\cH'$ is linear, 3-uniform and $(t-1)$-regular.



We will use the following special case of a theorem of Alon, Kim and Spencer \cite{AKS1997}.

\begin{thm}[\cite{AKS1997}]
\label{alonkimspencer}
Let $\cH'$ be a linear, 3-uniform, $(t-1)$-regular hypergraph on $N$ vertices. Then there exists a matching $M$ in $\cH'$ covering at least 

$$N-\frac{c_0N\ln^{3/2}(t-1)}{\sqrt{t-1}}$$ vertices, where $c_0$ is an absolute constant.

\end{thm}

Note that $\cH$ has  $$(1+o(1))\frac{1}{2}(t-1)^{1/2}n^{3/2}$$ vertices, thus the number of vertices in $\cH'$ is $$N= (1+o(1))\frac{1}{2}(t-1)^{1/2}n^{3/2}+o(n^{3/2}).$$ 

Applying Theorem \ref{alonkimspencer} we get a matching $M$ in $\cH'$. We delete at most $o(n^{3/2})$ hyperedges of $M$ that contain a vertex from $X$. This way we get a matching $M'$ in $\cH$ that covers all but 

$$\frac{c_0N\ln^{3/2}(t-1)}{\sqrt{t-1}}+o(n^{3/2})$$ vertices of $\cH$.  This implies, 

$$|M'| \ge \left( 1-\frac{c_0}{\sqrt{t-1}}\ln^{3/2}(t-1) \right) \frac{\sqrt{t-1}}{6}n^{3/2}+o(n^{3/2}),$$

\noindent
Finally, $ex_3(n, \{C_2, K_{2,t}\}) \ge |M'|$ -- indeed, by definition, $M'$ corresponds to a set of triangles in $G$ such that no two of them share an edge. So replacing them by hyperedges we get a $3$-uniform Berge-$K_{2,t}$-free linear hypergraph with $|M'|$ hyperedges, as desired. Note that the lower bound in Theorem \ref{linearlower2} does not have the additive term $o(n^{3/2})$ because we can choose $c$ in Theorem \ref{linearlower2} to be large enough (compared to $c_0$) so that the right hand side of the above inequality is at least the bound mentioned in our theorem.

\qed
\section{Remarks}

We finish this article with some questions and remarks concerning our results.

\vspace{2mm}

$\bullet$ In Corollary \ref{maincor} we provided an asymptotics for $ex_3(n,K_{2,t})$ for $t\ge 7$. It would be interesting to determine the asymptotics in the remaining cases. We conjecture the following.
\begin{conjecture}
For $t = 3,4,5,6$, we have $$ex_3(n,K_{2,t})=(1+o(1)) \frac{1}{6}(t-1)^{3/2}n^{3/2}.$$
\end{conjecture}

\vspace{2mm}

$\bullet$ In Theorem \ref{linearberge} and Theorem \ref{linearlower2} we showed that the asymptotics of $ex_3(n,\{C_2,K_{2,t}\})$ is close to being sharp for large enough $t$. However, it would be interesting to determine the asymptotics for all $t \ge 3$.

\vspace{2mm}

$\bullet$ In Theorem \ref{main}, we studied a class of $r$-uniform Berge-$F$-free hypergraphs. It would be interesting to extend these results to a larger class of hypergraphs. Similarly, it would be interesting to see if our  results in the linear case (in Section \ref{linear_case}) can be extended.

\vspace{2mm}

$\bullet$ Finally we note that there is a correspondence between Tur\'an-type questions for Berge hypergraphs and forbidden submatrix problems (for an updated survey of the latter topic see \cite{A2013}). For more information about this correspondence, see \cite{AS2016}, where they prove results about forbidding small hypergraphs in the Berge sense and they are mostly interested in the order of magnitude. Very recently, similar research was carried out in \cite{SS2017} and also see the references therein. We note that our results provide improvements of some special cases of Theorem 5.8. in \cite{SS2017}. 

\section*{Acknowledgement}

We thank the anonymous reviewers for their careful reading of our manuscript and their many insightful comments and suggestions improving the presentation of our article.

\ 

Research of Gerbner was supported by the J\'anos Bolyai Research Fellowship of the Hungarian Academy of Sciences and by the National Research, Development and Innovation Office -- NKFIH, grant K 116769.

\noindent
Research of Methuku was supported by the National Research, Development and Innovation Office -- NKFIH, grant K 116769.

\noindent
Research of Vizer was supported by the National Research, Development and Innovation Office -- NKFIH, grant SNN 116095.

\end{document}